\documentclass{article}
\usepackage{amsmath, latexsym, amsfonts, amssymb, amsthm, amscd}

\newtheorem {lemme} {Lemma} [section]
\newtheorem {theoreme} {Theorem} [section]
\newtheorem {proposition} {Proposition} [section]

\newtheorem {remarque} {Remark} [section]
\newtheorem {definition} {Definition} [section]
\newtheorem {exemple} {Example} [section]
\newcommand{\tr}{{\rm tr}}

\newcommand{\R}{\mathbb {R}}

\numberwithin{equation}{section}

\newcommand{\vers}{\mathop{\longrightarrow }}
\title{Outliers in the spectrum of large deformed unitarily invariant models}
\author{S. T. Belinschi, H. Bercovici, M. Capitaine, M. F\'evrier}
\begin{document}

\maketitle

\begin{abstract}
We investigate the asymptotic behavior of the eigenvalues of 
the sum $A_N+U_N^*B_NU_N$, where $A_N$ and $B_N$ are deterministic 
$N\times N$ Hermitian matrices having respective limiting compactly 
supported distributions $\mu$ and $\nu$, and $U_N$ is a random $N\times 
N$ unitary matrix distributed according to Haar measure. 
We assume that $A_N$ has a fixed number of fixed eigenvalues (spikes) 
outside the support of $\mu $ whereas the distances between the other 
eigenvalues of $A_N$ and the support of $\mu $, and between the 
eigenvalues of $B_N$ and the support of $\nu $ uniformly go to zero as 
$N$ goes to infinity. We establish that only a particular subset of the 
spikes will generate some eigenvalues of $A_N+U_N^*B_NU_N$ 
outside the support of the limiting spectral measure, called outliers. 
This phenomenon is fully described in terms of free probability 
involving the subordination function related to the free additive 
convolution of $\mu $ and $\nu $. Only finite rank 
perturbations had been considered up to now. 
\end{abstract}

\section{Introduction}

The set of possible spectra for the sum of two deterministic matrices $A_N$ and $B_N$ 
depends in complicated ways on the spectra of $A_N$ and $B_N$ (see \cite{Fulton98}). 
Nevertheless, if one adds some randomness to the eigenspaces and 
assumes them to be in generic position with respect to each other, 
when $N$ becomes large, free probability provides  a good understanding 
of the global behavior of the spectrum of the sum of matrices. Indeed,
if $X_N=A_N+U_N^*B_NU_N$, where $U_N$ is a Haar unitary random matrix 
(i.e. from the set of unitary matrices equipped with the normalized Haar measure as probability measure), 
if the spectral measures of $A_N$ and $B_N$ converge weakly towards 
respective compactly supported distributions $\mu $ and $\nu $, then 
building on the groundbreaking result of Voiculescu \cite{Voiculescu91}, Speicher proved in \cite{Speicher93a} 
the almost sure weak convergence of the spectral measure of $X_N$ to the free convolution $\mu \boxplus \nu $, 
which is known to be a compactly supported probability measure on $\mathbb{R}$.  
We refer the reader to \cite{VDN92} for an introduction to free probability theory. 

In \cite{BGRao09}, the authors investigated the case where $A_N$ has  a finite rank $r$ independent of $N$. 
More precisely, they considered a deterministic Hermitian perturbation matrix $A_N$ 
having $r$ non-zero eigenvalues $\gamma_1\geq \cdots \geq \gamma_s > 0 > \gamma_{s+1} \geq \cdots \geq \gamma_r$.
Note that in that case, $\mu\equiv \delta_0$ and the global limiting behavior
of the spectrum of $X_N$  is not affected by such a matrix $A_N$.
Thus,   the spectral measure of $X_N=A_N+U_N^*B_NU_N$ still converges 
to the limiting spectral measure of $B_N$. 
Nevertheless, in \cite{BGRao09}, the authors uncovered a phase transition phenomenon 
whereby the limiting value of the  extreme eigenvalues
of $X_N$ differs from that of $B_N$ if and only if 
the eigenvalues of $A_N$ are above a certain critical threshold:

\begin{theoreme}[Theorem 2.1 in \cite{BGRao09}]\label{BGRao}
Denote by $\lambda_1(X_N) \geq \cdots \geq \lambda_N(X_N)$ the ordered eigenvalues of $X_N$.
Let $a$ and $b$ be respectively the infimum and supremum of the support of $\nu$. 
Assume that the smallest and largest eigenvalue of $B_N$ converge almost surely to $a$ and $b$.
Then, we have for each $1\leq i \leq s$, almost surely, 
$$\lambda_i(X_N) \rightarrow_{N \rightarrow + \infty} \left\{ \begin{array}{ll} G_{\nu}^{-1}(1/\gamma_i) \mbox{~~if~~} \gamma_i > 1/\lim_{z\downarrow b} G_\nu( z),\\ b \mbox{~~otherwise}, \end{array}\right.$$
while for each fixed $i>s$,  almost surely, $\lambda_i(X_N) \rightarrow_{N \rightarrow + \infty} b.$
Similarly, for the smallest eigenvalues, 
we have for each $0\leq j <r-s$, almost surely, 
$$\lambda_{N-j}(X_N) \rightarrow_{N \rightarrow + \infty} \left\{ \begin{array}{ll} G_{\nu}^{-1}(1/\gamma_{r-j}) \mbox{~~if~~} \gamma_{r-j} < 1/\lim_{z\uparrow a} G_\nu( z),\\ a \mbox{~~otherwise}, \end{array}\right.$$
while for each fixed $j\geq r-s$,  almost surely, $\lambda_{N-j}(X_N) \rightarrow_{N \rightarrow + \infty} a.$
Here, $$G_\nu\colon\mathbb{C}\setminus \mathrm{supp}(\nu )\to \mathbb{C},\quad 
G_\nu(z)=\int _\mathbb{R}\frac{d\nu (t)}{z-t},$$
is the Cauchy-Stieltjes transform of $\nu $, $G_{\nu}^{-1}$ is its functional inverse.
\end{theoreme}

Note that \cite{BGRao09} lies in the lineage of recent works studying the
influence of some finite rank additive or multiplicative perturbations on the extremal eigenvalues of
classical random matrix models, the seminal paper being \cite{BBP05} 
where Baik, Ben Arous and P\'ech\'e pointed out the so-called BBP phase transition
(cf. \cite{John, BBP05, BaikSil06} for sample covariance matrices, 
\cite{FurKom81, Peche06,FePe, CDF09, PRS} for deformed Wigner models 
and \cite{LV} for information-plus-noise models). 
Such problems were first extended to non-finite rank perturbations 
in \cite{RaoSil09} and \cite{BaiYao08b} for sample covariance matrices 
and in \cite{CDFF10} for  deformed Wigner models. 
In this last paper, the authors pointed out that the subordination function 
relative to the free additive convolution of a semicircular distribution
with the limiting spectral distribution of the perturbation plays an important role 
in the fact that some eigenvalues of the deformed Wigner model separate from the bulk.
Note that in \cite{Capitaine11}, the author explained how the results of \cite{RaoSil09} and \cite{BaiYao08b} 
in the sample covariance matrix setting can also be described in terms of free probability
involving the subordination function related to the free multiplicative
convolution of a Marchenko-Pastur distribution with the limiting spectral distribution
of the multiplicative perturbation.

In this paper, we investigate the asymptotic spectrum of the model $X_N=A_N+U_N^*B_NU_N$ 
without the extra requirement that one of the limiting spectral measures $\mu$
and $\nu$ of the respective deterministic Hermitian matrices $A_N$ and $B_N$ is a point mass. 
We assume that $A_N$ has a fixed number $r$ of fixed eigenvalues (spikes) 
outside the support of $\mu $ whereas the distances between the other eigenvalues of $A_N$ 
and the support of $\mu $, and between the eigenvalues of $B_N$ and the support of $\nu $ 
uniformly go to zero as $N$ goes to infinity. 
We are interested in the following questions: are some of the eigenvalues of $X_N$ 
almost surely separating from the bulk, that is being located outside
the support of the limiting spectral distribution $\mu \boxplus \nu $? What is 
the  characterization  of the spiked eigenvalues of $A_N$ 
that generate such outliers in the spectrum of $X_N$? 
Is there an interpretation of these phenomena in terms of the subordination functions 
related to the free additive convolution of $\mu$ and $\nu$?

In the particular case where $r=0$, it is proved in \cite{ColMal11} 
that, almost surely, for large enough $N\in \mathbb{N}$, 
all eigenvalues of $X_N$ are included in a neighborhood of 
the support of the limiting spectral distribution $\mu \boxplus \nu $.
In the following, we will therefore assume that $r\geq 1$.
Thus, this paper may be seen as an extension of \cite{BGRao09} 
since it extends the framework of \cite{BGRao09} to non-finite rank perturbations
but also as  an extension of \cite{CDFF10} since it extends 
the free probabilistic interpretation of outliers phenomena 
in terms of subordination functions described  in \cite{CDFF10}
for Wigner deformed models to deformed unitarily invariant models.
Here, the characterization in terms of subordination functions of the spiked eigenvalues 
of $A_N$ that generate  outliers in the spectrum of $X_N$ turns out to be 
more complex that the one presented in \cite{CDFF10} for deformed Wigner models, 
but it is a completely natural extension as we explain in Remark \ref{infinidiv}. 
It is worth noticing that we uncover here a new phenomenon: 
a single  spiked eigenvalue of $A_N$ may generate asymptotically 
a finite or countably infinite set of outliers of $X_N$. 
This comes from the fact that the restriction to the real line of some subordination functions 
may be many-to-one, unlike the subordination function 
related to free convolution with a semicircular distribution 
studied in \cite{CDFF10}.

The approach of the proof of our main result (i.e Theorem \ref{mainresult}) 
is in the spirit of \cite{BGRao09} and comes down to prove the almost sure convergence 
of a certain $r \times r$ matrix, involving the resolvent of the deformation of $U_N^*B_NU_N$
by some matrix $A_N'$ without spikes. Its almost sure convergence is proved by establishing 
an approximate matricial subordination equation and using  a concentration argument.
Then, the problem is reduced to solving an equation involving the spikes 
and the subordination function related to the free convolution of $\mu$ and $\nu$. 

The paper is organized as follows. In Section 2, we introduce the additive deformed models 
we consider in this paper; we also introduce some basic notations that will be used throughout the paper.
Section 3 is devoted to definitions and results concerning free convolution and subordination functions, 
some of them being necessary to state our main result Theorem \ref{mainresult} in Section 4. 
The proof of  Theorem \ref{mainresult} is presented in Sections 4 and 5. 
More precisely, we explain in Section 4 how the proof comes down to prove the almost sure convergence 
of a $r \times r$ matrix and Section 5 deals with the proof of this convergence.

\section{Notations and presentation of the model}

\noindent Throughout this paper, we will use the following notations.
\begin{itemize}
\item[-] $\mathbb{C}^+$ will denote the complex upper half-plane $\{z \in \mathbb{C}, \, \Im z >0\}$. Similarly,
$\mathbb{C}^-$ will stand for
$\{z \in \mathbb{C}, \, \Im z <0\}$.
\item[-] We will denote by $M_{m}(\mathbb{C})$ the set of $ m\times m$ matrices with complex entries and $GL_m(\mathbb{C})$ the subset of invertible ones.
$\Vert  ~ \Vert $ will denote the operator norm.
\item[-]  For any matrix $M$, we will denote its kernel by Ker($M$).
\item[-] $E_{ij}$ stands for the  matrix such that $(E_{ij})_{kl}=\delta_{ik}\delta_{jl}$.
\item[-] For any $N\times N$ Hermitian matrix $M$, we will denote by 
$$\lambda_1(M) \geq \ldots \geq \lambda_N(M)$$
\noindent its ordered eigenvalues.
\item[-] For a probability measure $\tau $ on $\R$, 
we denote by $\mbox{supp}(\tau)$ its topological support. 
\item[-] $C, C_1, C_2, C'$ denote nonnegative constants which may vary from line to line.
\end{itemize}

In this note, we consider the following model $X_N=A_N+U_N^*B_NU_N$, where: 

\begin{itemize}
\item $A_N$ is a deterministic $N\times N$ Hermitian matrix whose spectral measure 
$\mu _{A_N} := \frac{1}{N} \sum_{i=1}^N \delta _{\lambda _i(A_N)}$ 
weakly converges to some compactly supported probability measure $\mu $ on $\mathbb{R}$. 
We assume that there exists a fixed integer $r\geq 0$ (independent from $N$) 
such that $A_N$ has $N-r$ eigenvalues $\alpha _j^{(N)}$ satisfying 
$$\max _{1\leq j\leq N-r} \mathrm{dist}(\alpha _j^{(N)},\mathrm{supp}(\mu ))\vers _{N \rightarrow \infty } 0.$$ 
We also assume that there are $J$ fixed real numbers $\theta _1 > \ldots > \theta _J$ 
independent of $N$ which are outside the support of $\mu $ and such that each $\theta _j$ 
is an eigenvalue of $A_N$ with a fixed multiplicity $k_j$ (with $\sum_{j=1}^J k_j=r$). 
The $\theta_j$'s will be called the spikes or the spiked eigenvalues of $A_N$. 
\item $B_N$ is a deterministic $N\times N$ Hermitian matrix whose spectral measure 
$\mu _{B_N} := \frac{1}{N} \sum_{i=1}^N \delta _{\lambda _i(B_N)}$ 
weakly converges to some compactly supported probability measure $\nu $ on $\mathbb{R}$. 
We assume that the eigenvalues $\beta _j^{(N)}$ of $B_N$ satisfy 
$$\max _{1\leq j\leq N} \mathrm{dist}(\beta _j^{(N)},{\rm supp}(\nu ))\vers _{N \rightarrow \infty } 0.$$
\item $U_N$ is a random $N\times N$ unitary matrix distributed according to Haar measure. 
\end{itemize} 

\section{Free convolution}\label{freeconv}

Free  convolution appears as a natural analogue of the classical convolution 
in the context of free probability theory.
Denote by ${\cal M}$ the set of Borel probability measures 
supported on the real line. 
For $\mu$ and $\nu$ in ${\cal M}$ one defines the free additive convolution 
$\mu \boxplus \nu$ of $\mu$ and $\nu$ as the distribution of 
$X+Y$ where $X$ and $Y$ are free self-adjoint random variables with distribution $\mu$ and $\nu$. 
We refer the reader to \cite{VDN92} for an introduction to free probability theory
and to \cite{Voiculescu86} 
and \cite{BercoviciVoiculescu} 
for free convolution. In this section, we recall the analytic approach 
developed in \cite{Voiculescu86} to calculate the free convolution of measures, 
we present the important subordination property and some related  
fundamental results we will refer to later.

\subsection{Additive Free convolution}

For any positive 
finite Borel measure $\tau $ on $\mathbb{R}$, 
the Cauchy-Stieltjes transform of $\tau $
$$G_\tau\colon\mathbb{C}\setminus \mathrm{supp}(\tau )\to \mathbb{C},\quad 
G_\tau(z)=\int _\mathbb{R}\frac{d\tau (t)}{z-t}$$
is analytic, maps the upper half-plane $\mathbb{C}^+$ into the
lower half-plane $\mathbb{C}^-$ and satisfies the conditions $G_\tau(
\overline{z})=\overline{G_\tau(z)},$ 
$\lim_{y\to +\infty }iyG_\tau (iy)=\tau(\mathbb{R})$. These conditions in fact 
characterize functions which are Cauchy-Stieltjes transforms of positive
finite measures. For probability measures $\tau $ with compact support,
$G_\tau $ is analytic on the neighbourhood of infinity
$\{ z\in \mathbb{C}\colon |z| > \max \{ |x|\colon x\in \mathrm{supp}(\tau )\} \} $.
The Cauchy-Stieltjes transform allows us to recover 
the measure $\tau $ as the weak${}^*$-limit 
$$d\tau (x)=\lim _{y\to 0}\frac{-1}{\pi }\Im G_\tau (x+iy).$$
For the absolutely continuous part of $\tau $, the situation is better:
the relation $$\frac{d\tau }{dx}=\lim _{y\to 0}\frac{-1}{\pi }\Im G_\tau (x+iy)$$
holds Lebesgue-a.e. as equality of functions. 
For some purposes, it is convenient to work with the reciprocal Cauchy-Stieltjes transform, 
which is the analytic self-map of the upper half-plane defined by: 
$$\forall z \in \mathbb{C}^+,\ F_\tau(z)=\frac{1}{G_\tau(z)}.$$ 
The Cauchy-Stieltjes transform of a compactly supported 
probability measure $\tau $ is invertible in the neighborhood of infinity, 
with functional inverse, denoted by $G_{\tau}^{-1}$, defined in a neighborhood of $0$. 
Define then the R-transform of $\tau $ 
by: $$R_\tau (z) = G_{\tau}^{-1}(z) - \frac{1}{z}.$$
Given two compactly supported probability measures $\mu $ and $\nu $, 
there exists a unique probability measure $\lambda $ such that
$$R_\lambda = R_\mu + R_\nu $$
on a domain where these functions are defined. 
The probability measure $\lambda $ is called 
the free additive convolution of $\mu $ and $\nu $ and denoted by $\mu \boxplus \nu $. 
The support of the probability measure $\mu \boxplus \nu $ being a compact set, 
we will denote it by $$K:=\mathrm{supp}(\mu \boxplus \nu ).$$
Moreover, given $\varepsilon > 0$, we will use the following notations: 
$$K_\varepsilon ^\mathbb{R}:=\{x\in \mathbb{R}~|~d(x,\mathrm{supp}(\mu \boxplus \nu)\leq \varepsilon \},$$
$$K_\varepsilon ^\mathbb{C}:=\{z\in \mathbb{C}~|~d(z,\mathrm{supp}(\mu \boxplus \nu)\leq \varepsilon \}.$$

\subsection{Free subordination phenomenon}

We recall the subordination phenomenon for the Cauchy-Stieltjes 
transform of the free additive convolution of measures. 
Given Borel probability measures $\mu $ and $\nu $ on $\mathbb{R}$, 
the Cauchy-Stieltjes transform of the free additive convolution $\mu \boxplus \nu$
is subordinated to the Cauchy-Stieltjes transform of any of $\mu $ or $\nu $: 
there are two analytic self-maps of the upper half-plane $\omega_1,\omega_2:\mathbb{C}^+\to \mathbb{C}^+$
such that: $$\forall z\in \mathbb{C}^+,\ G_{\mu \boxplus \nu}(z)=G_\mu (\omega_1(z))=G_\nu(\omega_2(z)).$$
The subordination maps $\omega_1,\omega_2$ are also related by: 

\begin{equation}\label{subordination3}
\forall z\in \mathbb{C}^+,\ \omega_1(z)+\omega_2(z)=z+F_{\mu \boxplus \nu}(z). 
\end{equation}
Moreover, for $j\in \{1,2\}$, $$\lim_{y\rightarrow +\infty}\frac{\omega_j(iy)}{iy}=1.$$
It then follows, using Nevanlinna representation of analytic self-maps of the upper half-plane, 
that: $$\forall z\in \mathbb{C}^+,\ \Im \omega_j(z)\geq \Im z;$$
equality can occur only when one of the measures $\mu , \nu $ is a point mass 
(more specifically, $\exists z\in \mathbb{C}^+,\ \Im\omega_1(z)=\Im z\iff 
\forall z\in \mathbb{C}^+,\ \Im \omega_1(z)=\Im z \iff \nu$ is a 
point mass, and $\exists z\in \mathbb{C}^+,\ \Im\omega_2(z)=\Im z\iff 
\forall z\in \mathbb{C}^+,\ \Im \omega_2(z)=\Im z \iff \mu$ is a point mass).
These results, first obtained in full generality in \cite{Biane98}, have been given a new interpretation 
in terms of Denjoy-Wolff points of analytic functions in \cite{BelBer07}. 
We will state the upper half-plane version of the Denjoy-Wolff
theorem below:

If $f\colon\mathbb C^+\to\mathbb C^+$ is analytic and not a M\"obius
transformation, then only one of the
following three cases can occur:
\begin{enumerate}
\item there exists a unique point $\omega\in\mathbb C^+$ so that
$f(\omega )=\omega $ and $|f'(\omega)|<1$. Then the iterations
$f^{\circ n}$ converge uniformly on compacts to the constant function
taking the value $\omega$; 
\item there exists a unique $\omega\in\mathbb R$ so that 
$\lim_{r\downarrow0}f(ir+\omega)=\omega$ {\em and} 
$$0<\lim_{r\downarrow0}\frac{f(ir+\omega)-\omega}{ir}\leq 1.$$
Then the iterations
$f^{\circ n}$ converge uniformly on compacts to the constant function
taking the value $\omega$;
\item $\lim_{r\uparrow+\infty}f(ir)=\infty$ {\em and} 
$$1\leq\lim_{r\uparrow+\infty}\frac{f(ir)}{ir}<\infty.$$
Then the iterations
$f^{\circ n}$ converge uniformly on compacts to infinity.
\end{enumerate}
The point $\omega$ in situations 1. and 2. (and infinity in 3.)
is called the Denjoy-Wolff point of $f$.
Note that the second limit in 2. above always exists in $(0;+\infty]$ as soon as $\lim_{r\downarrow0}f(ir+\omega)=\omega $, 
and is called the Julia-Carath\'eodory derivative of $f$ at $\omega $. 
Denjoy and Wolff proved that any analytic function $f
:\mathbb{C}^+ \rightarrow{\mathbb{C}^+}$ as above 
has a Denjoy-Wolff point
(see \cite{GarnettBook} for details). 
Using this result, a new proof of Biane's subordination result for free convolution was given in \cite{BelBer07}, 
by identifying $\omega _1(z)$ as the Denjoy-Wolff point of the function 
$$f_z(\omega ):=F_\nu(F_\mu (\omega )-\omega +z)-(F_\mu (\omega )-\omega+z )+z.$$

We will need the following lemma collecting results on extensions of the subordination maps: 

\begin{lemme} \label{extension}
For $j\in\{1,2\}$, the function $\omega _j$, defined on $\mathbb{C}^+$, 
has an extension (still denoted by $\omega _j$) to $\mathbb{C}$ so that: 
\begin{enumerate}
\item[(a)] $\omega _j$ is continuous on $\mathbb{C}^+\cup \mathbb{R}$;
\item[(b)] $\omega _1(\{\infty\}\cup\mathbb{R}\setminus \mathrm{supp}(\mu \boxplus \nu )) \subseteq\{\infty\}\cup \mathbb{R}\setminus \mathrm{supp}(\mu );$
\item[(b')] $\omega _2(\{\infty\}\cup\mathbb{R}\setminus \mathrm{supp}(\mu \boxplus \nu )) \subseteq\{\infty\}\cup \mathbb{R}\setminus \mathrm{supp}(\nu );$
\item[(c)] $\forall z\in \mathbb{C}\setminus \mathbb{R}, \overline{\omega_j(z)}=\omega_j(\overline{z});$
\item[(d)] $\omega _j$ is meromorphic on 
$\mathbb{C}\setminus \mathrm{supp}(\mu \boxplus \nu )$.
\end{enumerate}
\end{lemme}

\noindent {\bf Proof.} 
As noted in \cite[Theorem 3.3]{Belinschi08}, $\omega_j|_{\mathbb C^+}$ 
has a continuous extension to the real line. This proves (a). 
Letting $z$ tend to $x\in \mathbb{R}\setminus \mathrm{supp}(\mu \boxplus \nu )$ in \eqref{subordination3}, 
one notices that $\omega_j(x)$ necessarily belongs to the boundary of
$\mathbb C^+$, because so does $F_{\mu \boxplus \nu}(x)$. To be more
specific, if $F_{\mu \boxplus \nu}(x)\in\mathbb R$, then $\omega_j(x)
\in\mathbb R$ by relation \eqref{subordination3} and the Nevanlinna
representation.
If $F_{\mu \boxplus \nu}(x)=\infty,$ then $x$ is an isolated simple pole
of $F_{\mu \boxplus \nu}$, and the same \eqref{subordination3}
tells us that either $\omega_1(x)=\infty$ and $\omega_2(x)=
-m_1(\mu)+x$, or $\omega_2(x)=\infty$ and $\omega_1(x)=-m_1(\nu)+x$,
and thus $x$ is an isolated simple pole of exactly one of the two $\omega_j$.
(Here $m_1(\tau)$ denotes the first moment of $\tau$.)
This proves (b), (b') and (d). 
We then extend $\omega_j$ to $\mathbb{C}^-$ by requiring: 
$$\forall z\in \mathbb{C}^-, \omega_j(z)=\overline{\omega_j(\overline{z})},$$
as in \cite{Biane98}, so that (c) holds.
$\Box$\\

\noindent 
It should be noted that the proof from \cite{BelBer07} shows that
$\omega_1(z)$ is the Denjoy-Wolff point of $f_z$ only for $z\in
\mathbb C^+$. Indeed, it is still an open problem whether this is true
for all $z\in\mathbb C^+\cup\mathbb R$. However, if 
$x\in\mathbb R$ is so that $\omega_1$ is analytic in $x$ and
real on some interval around $x$, the Julia-Carath\'eodory Theorem 
requires $\omega_1'(x)\in(0,+\infty)$. 
As noted in the above lemma, this implies that
$F_\mu$ is real meromorphic around $\omega_1(x)$ and $F_\nu$ real 
analytic around $\omega_2(x)=F_\mu(\omega_1(x))-\omega_1(x)+x$.
Thus, $f_x$ will be real analytic around $\omega_1(x)$.
Taking limits shows that $f_x(\omega_1(x))=\omega_1(x)$
and $\omega_1'(x)=\partial_xf_x(\omega_1(x))\left[1-\partial_\omega
f_x(\omega_1(x))\right]^{-1}.$
Since $\partial_xf_x(\omega_1(x))=F_\nu'(\omega_2(x))>0$, in
order for $\omega_1'(x)>0$ as required by the Julia-Carath\'eodory
Theorem, we must have $\partial_\omega
f_x(\omega_1(x))<1$, and hence $\omega\mapsto f_x(\omega)$ has
$\omega_1(x)$ as a fixed point in which the derivative is 
less than one. This is the case 2. of the upper half-plane
version of the Denjoy-Wolff theorem given above. Thus, in this
case $\omega_1(x)$ is still necessarily the {\em unique}
Denjoy-Wolff point of $\omega\mapsto f_x(\omega)$.

Let us give a slightly different formulation for the results
of \cite{BelBer07} concerning the subordination functions, more
appropriate to the needs of our paper.
Assume neither $\mu $ nor $\nu$ is a point mass and denote 
$$h_\mu(z)=F_\mu(z)-z,\quad h_\nu(z)=F_\nu(z)-z.$$
We re-write the ideas of \cite{BelBer07}, where 
$\omega_1$ and $\omega_2$ are
identified as the Denjoy-Wolff points of self-maps of 
$\mathbb C^+$ indexed by $z$ (see above), but with a 
formulation chosen to make the statement somehow
independent of complex dynamics. It follows from \cite{BelBer07}
that $\omega_1$ and $\omega_2$ are
identified uniquely by the following system of equations
$$\left\{
\begin{array}{lcr} 
\omega_1(z)-h_\nu(\omega_2(z))&=&z\\
\omega_2(z)-h_\mu(\omega_1(z))&=&z
\end{array}\right..
$$
We can look upon this system as an implicit equation for a two-variable map:
we define
$f(w_1,w_2,z)=(w_1-h_\nu(w_2)-z,w_2-h_\mu(w_1)-z)$.
A straightforward application of the implicit function theorem
indicates that $h_\mu'(w_1)h_\nu'(w_2)=1$ is the only obstacle to
an analytic solution $z\mapsto(\omega_1(z),\omega_2(z))$
to the equation $f(\omega_1(z),\omega_2(z),z)=(0,0)$.
The analysis in \cite{BelBer07} shows that 
$$\forall z\in\mathbb C^+,\ |h_\mu'(\omega_1(z))h_\nu'(\omega_2(z))|<1.$$ 
This is a consequence of the fact that non-trivial 
- i.e. not M\"{o}bius - maps have derivatives less than one in their 
Denjoy-Wolff points. Since the functions $h$ and $\omega$ map 
the parts of $\mathbb R$ which lay in their domains of analyticity 
in $\mathbb R$ and their derivatives on $\mathbb R$ are necessarily
positive, we conclude that 
$0<h_\mu'(\omega_1(x))h_\nu'(\omega_2(x))<1$ 
for any $x\in\mathbb R$ in the domain of analyticity of $\omega_1
$ and $\omega_2$. In particular, 
$0<h_\mu'(\omega_1(x))h_\nu'(h_\mu(\omega_1(x))+x)<1,$
and the obstacle to an analytic solution $\omega_1$ is described 
by the equality $h_\mu'(\omega_1(x))h_\nu'(h_\mu(\omega_1(x))+x)
=1$.

\begin{theoreme}\label{outlier0}
Assume that neither $\mu $ nor $\nu$ is a point mass. 
Given $\theta \in \mathbb{R}\setminus \mathrm{supp}(\mu )$, then 
$\rho \in \mathbb{R}$ is a solution of the equation $$\omega_1 (\rho )=\theta $$ 
belonging to $\mathbb{R}\setminus \mathrm{supp}(\mu \boxplus \nu )$ 
if and only if $h_\mu(\theta)+\rho$ belongs to the domain of analyticity of $h_\nu$,
and $\rho$ is a solution of:
\begin{equation}\label{condition}
h_\nu(h_\mu(\theta)+\rho)-\theta+\rho=0,\quad 0<h_\mu'(\theta)h_\nu'(h_\mu(\theta)+\rho)<1.
\end{equation}
\end{theoreme} 

{\bf Proof:} Let $\rho \in \mathbb{R}\setminus \mathrm{supp}(\mu \boxplus \nu )$ 
be a solution of the equation $\omega_1 (\rho )=\theta $. Taking the limit $z\rightarrow \rho $ 
in the second equality of the system above, we obtain 
$$h_\mu(\theta)+\rho=\omega_2(\rho)$$ 
which belongs to $\{\infty\}\cup\mathbb{R}\setminus \mathrm{supp}(\nu )$ by Lemma \ref{extension} (b'), 
and therefore to the domain of analyticity of $h_\nu$. 
Taking the limit $z\rightarrow \rho $ in \eqref{subordination3}, one gets 
$$h_\nu(h_\mu(\theta)+\rho)=h_\nu(\omega_2(\rho))=\omega_1(\rho)-\rho =\theta -\rho.$$ 
Since $\rho \in \mathbb{R}\setminus \mathrm{supp}(\mu \boxplus \nu )$, 
$\omega_1$ is analytic at $\rho $ (see Lemma \ref{extension} (d) - by
hypothesis $\omega_1(\rho)=\theta\neq\infty$), 
and there must be no obstacle to the existence of an analytic solution around $\rho $ as in the discussion above. 
Hence $$0<h_\mu'(\theta)h_\nu'(h_\mu(\theta )+\rho)<1.$$
Conversely, let $\rho \in \mathbb{R}$ be such that 
$h_\mu(\theta)+\rho$ belongs to the domain of analyticity of $h_\nu$,
and $\rho$ is a solution of \eqref{condition}. 
This in particular implies that $\theta$ is the Denjoy-Wolff point
of $f_\rho(\omega)=h_\nu(h_\mu(\omega)+\rho)+\rho$, $\omega\in
\mathbb C^+$,
as $f_\rho(\theta)=\theta$ and $f_\rho'(\theta)\in ]0,1[$.
An application of the implicit function theorem shows that the
dependence of $\theta$ on the complex variable $\rho$ is analytic
around the given real point, and has a positive derivative.
Analytic continuation, the uniqueness of the Denjoy-Wolff point 
and the considerations above guarantee that $\omega_1(\rho)=\theta$.
$\Box$\\

\begin{remarque}
We should note that the set of points $\rho$ that satisfy the equality from 
\eqref{condition}, being the set of zeroes of an analytic map, is necessarily discrete, 
by the principle of isolated zeroes. Therefore the set of solutions of $\omega_1(\rho )=\theta $ 
in $\mathbb{R}\setminus \mathrm{supp}(\mu \boxplus \nu )$ is discrete as well. 
This set of solutions may be empty, finite or countably infinite, as illustrated in the next remarks.
\end{remarque}

\begin{remarque}\label{infinidiv}
If $\nu $ is $\boxplus $-infinitely divisible, it is known (see \cite{Biane97b}) 
that $\omega _1$ is a conformal bijection between $\mathbb{C}^+$ 
and a simply connected domain $\Omega \subseteq \mathbb{C}^+$, whose inverse 
is the restriction to $\Omega $ of the map $H$ defined on $\mathbb{C}^+$ by:
$$H(z):=z+R_\nu(G_\mu(z)),$$
where $R_\nu$ is the R-tranform of the measure $\nu$. Under this extra hypothesis, 
the content of Theorem \ref{outlier0} reads simply: 
$\rho \in \mathbb{R}$ is a solution of the equation $$\omega_1 (\rho )=\theta $$ 
belonging to $\mathbb{R}\setminus \mathrm{supp}(\mu \boxplus \nu )$ 
if and only if 
\begin{equation}\label{conditionID}
\rho=H(\theta), \quad H'(\theta ) > 0.
\end{equation}
Indeed, if $\rho \in \mathbb{R}$ satisfy \eqref{conditionID}, 
notice first that $$h_\mu(\theta)+\rho=F_\mu(\theta)+R_\nu(G_\mu(\theta))=G_\nu^{-1}(G_\mu(\theta)),$$
which is in the domain of analyticity of $h_\nu$ (recall that
$\theta\notin\mathrm{supp}(\mu)$ by hypothesis and thus
$G_\nu(h_\mu(\theta)+\rho)=G_\mu(\theta)$). 
Then $$h_\nu(h_\mu(\theta)+\rho)+\rho=F_\nu(G_\nu^{-1}(G_\mu(\theta)))-h_\mu(\theta)=F_\mu(\theta)-h_\mu(\theta)=\theta.$$
Finally, because of the relation $$H(\theta )=\theta-h_\nu(h_\mu(\theta)+\rho),$$ 
the condition $H'(\theta ) > 0$ implies the inequality in \eqref{condition}.\\
Conversely, assume that $\rho \in \mathbb{R}$ is such that 
$h_\mu(\theta)+\rho$ belongs to the domain of analyticity of $h_\nu$ and satisfies \eqref{condition}. 
The condition \eqref{conditionID} being equivalent to 
\begin{equation}
\theta=\rho -R_\nu(G_{\mu\boxplus\nu}(\rho)), \quad H'(\theta ) > 0,
\end{equation}
it is possible to check
that $\rho -R_\nu(G_{\mu\boxplus\nu}(\rho))$ is the unique Denjoy-Wolff 
point of $f_\rho$, which yields the equality above. The inequality follows immediately 
from the inequality in \eqref{condition}.
Note that there is at most one solution belonging to $\mathbb{R}\setminus \mathrm{supp}(\mu \boxplus \nu )$ which satisfies
the equation $$\omega_1 (\rho )=\theta $$ when $\nu $ is $\boxplus $-infinitely divisible.
\end{remarque}

\begin{exemple}
We should note that there are many examples in which the
restriction to the real line of subordination functions 
are many-to-one on their domain of analyticity. A
trivial example comes from free Brownian motions, in which the
spikes are associated to the matrix approximating the semicircular
distribution (note the difference from the problem studied
in \cite{CDFF10}!). 
Consider a Bernoulli distribution $b=(\delta_{-1}+\delta_1)/2$
and a standard semicircular $\gamma_t$ of variance $t$. We know that 
$G_{b\boxplus\gamma_t}(z)=G_{\gamma_t}(\omega_1(z))=G_b(\omega_2(z))$,
where $\omega_2$ satisfies 
$$
\omega_2(z)=z-tG_b(\omega_2(z)),\quad z\in\mathbb C^+\cup\mathbb R.
$$
We claim that for $t>0$ small enough, there are real values
taken twice by $\omega_1$. Indeed,
$$
h_b(h_{\gamma_t}(\theta)+\rho)=\frac{-1}{h_{\gamma_t}(\theta)+\rho}=
\frac{-2}{-\theta+\sqrt{\theta^2-4t}+2\rho},
$$
and then 
$$
h_b(h_{\gamma_t}(\theta)+\rho)=\theta-\rho\iff
2\rho^2+(\sqrt{\theta^2-4t}-3\theta)\rho+\theta^2-\theta\sqrt{\theta^2-4t}-2=0.
$$
The two solutions are
$$
\rho_{1,2}=\frac{3\theta-\sqrt{\theta^2-4t}\pm\sqrt{2\theta^2+2\theta
\sqrt{\theta^2-4t}+16-4t}}{4}.
$$
For these solutions to be real, we must have $|\theta|\ge2\sqrt{t}$ and
$$
\theta^2+8-2t\ge-\theta\sqrt{\theta^2-4t}.
$$
For our example we shall consider a positive spike, and so the requirements translate into $\theta\ge 2\sqrt{t}$ and $\theta^2
\ge2t-8$. For $t\le1$, this reduces to $\theta\ge2\sqrt{t}$.
On the other hand, the condition $h_{\gamma_t}'(\theta)h_b'
(h_{\gamma_t}(\theta)+\rho)<1$ means 
$$
\frac{h_{\gamma_t}'(\theta)}{(h_{\gamma_t}(\theta)+\rho)^2}<1.
$$
We recall from above that $\frac{-1}{h_{\gamma_t}(\theta)+\rho}=\theta
-\rho$. Thus this inequality is simply $h_{\gamma_t}'(\theta)(\rho-\theta)^2<1.$ 
We note that $h_{\gamma_t}'(\theta)=\frac{\theta}{2\sqrt{\theta^2-4t}}-\frac12\in]0,1[$ 
for all $\theta>3\sqrt{t}$. Moreover, $\lim_{\theta\to+\infty}\theta^2 
h_{\gamma_t}'(\theta)=2t$, and, again when $\theta\to+\infty$, the two
values of $\rho$ tend to infinity at a speed of the order of 
$\theta$ and at zero with a speed of order $\frac{2t}{\theta}$, respectively. 
Now we conclude easily: for $\theta$ large enough and 
$t<\frac{1}{2}$ (strict inequality!), both conditions in \eqref{condition} are satisfied for the existence of $\rho$,
and thus we obtain two spikes explicitely computed.
\end{exemple}

\begin{remarque}\label{pointmasses}
In the case, not covered by Theorem \ref{outlier0}, where $\mu$ or $\nu$ is a point mass, 
$\omega_1$ is equal either to $F_\nu$ or to a real translation. 
In the second case, the equation $$\omega_1 (\rho )=\theta $$ is trivial. In the first case, 
$F_\nu$ establishing continuously increasing bijections between each connected component 
of the complement of the support of $\nu $ and some subsets of $\mathbb{R}$, 
the set of solutions may be determined by real analysis: for example, 
if $\nu=\sum2^{-n}\delta_{1/n}$, then there are countably many open
intervals included in $]0,1[$ which are mapped bijectively onto $
\mathbb R$ by $F_\nu$.
\end{remarque}

In this paper we shall be concerned almost exclusively with $\omega_1$,
so for simplicity we shall adopt the
$$
\text{\bf Notation convention:}\quad\quad \omega_1=\omega.
$$

\section{Main results and sketch of proof}

\subsection{Main result and examples}

\begin{definition}
For each $j\in \{1, \ldots , J\}$, define $O_j$ 
the set of solutions in $\mathbb{R}\setminus \mathrm{supp}(\mu \boxplus \nu )$ 
of the equation
\begin{equation} \label{outlier}
\omega (\rho )=\theta _j,
\end{equation}
and $$O=\bigcup_{1\leq j\leq J}O_j.$$
\end{definition}

\noindent 
Recall that the sets $O_j$ defined above may be empty, finite, or countably infinite.

\begin{theoreme} \label{mainresult}
Denote by $\mathrm{sp}(X)$ the spectrum of the operator $X$.
The following results hold almost surely:
\begin{itemize}
\item  for each $\rho \in O_j$, for all  small enough $\varepsilon>0$, for all large enough $N$,
$$\mathrm{card}\{\mathrm{sp}(X_N)\bigcap ]\rho-\epsilon; \rho+\epsilon[\}=k_j;$$
\item for almost all $\eta>0$, for all small enough $\varepsilon > 0$, for large enough $N$, 
$$\mathrm{sp}(X_N)\bigcap{\mathbb{C}}\setminus K_{\eta }^{\mathbb{R}}\subset \bigcup_{\rho \in O\bigcap{\mathbb{C}}\setminus K_{\eta }^{\mathbb{R}}} ]\rho-\epsilon; \rho+\epsilon[.$$
\end{itemize}
\end{theoreme}

\begin{remarque}
The proof of this theorem, starting at the end of this section, 
and completed in the next one, works without any change 
when the spikes may depend on $N$, more precisely 
if we assume that the spectrum of $A_N$ consists of 
$N-r$ eigenvalues $\alpha _j^{(N)}$ satisfying 
$$\max _{1\leq j\leq N-r} \mathrm{dist}(\alpha _j^{(N)},\mathrm{supp}(\mu ))\vers _{N \rightarrow +\infty } 0,$$
and $r$ spikes $\theta _1^N \geq \ldots \geq \theta _r^N$ 
such that: $$\forall i\in \{k_1+\ldots+k_{j-1}+1,\ldots ,k_1+\ldots+k_j\},\ 
\theta _i^N \vers _{N \rightarrow +\infty } \theta_j.$$
The only reason we chose not to write the proofs under these slightly more general hypotheses 
is to avoid confusion in the notations.
\end{remarque}

\begin{remarque}
Actually, the conclusion of the preceding theorem holds for a random matrix $X_N=A_N+\tilde{B}_N$, 
where $A_N$ and $\tilde{B}_N$ are independent random Hermitian matrices satisfying almost surely the assumptions 
given in the first section and in the preceding remark, 
under the extra assumption that the distribution of the random matrix $\tilde{B}_N$ 
is invariant by conjugation by any unitary matrix in $\mathbb{U}_N$. Note however that the quantities 
$J, k_1, \ldots , k_J, \theta_1, \ldots , \theta _J$ are supposed deterministic. 
In particular, one recovers results from \cite{BGRao09}, \cite{CDF09} and \cite{CDFF10}. 
It is clear that a  random Hermitian matrix $\tilde{B}_N$ whose distribution 
is invariant by conjugation by any unitary matrix in $\mathbb{U}_N$ 
has the same distribution as $U_N^*B_NU_N,$
where $B_N$ is a random real diagonal matrix, 
and $U_N$ is a random unitary matrix distributed according to Haar measure 
and independent of $B_N$. 
The proof would then proceed, on the almost sure event on which $A_N$ and $\tilde{B}_N$ 
satisfy all the assumptions required, by conditioning with respect to the sigma-field generated 
by the sequences of random matrices $A_N, B_N$, and then applying Theorem \ref{mainresult}. 
\end{remarque}

\begin{exemple}
When the rank of $A_N$ remains finite, one recovers, using the preceding remark, 
results on the extremal eigenvalues from \cite{BGRao09} recalled in Theorem \ref{BGRao} 
(see also the Gaussian case in \cite{CDF09}). 
Indeed, in this setting, $\mu\equiv \delta_0$, and the set $O$ is therefore 
the set of solutions $\rho \in \mathbb{R}\setminus \mathrm{supp}(\nu )$ 
of equations $$F_\nu(\rho )=\theta _j,$$
as explained in Remark \ref{pointmasses}. Theorem \ref{mainresult} 
implies the conclusion of Theorem \ref{BGRao}, the condition 
$$\gamma_i > 1/\lim_{z\downarrow b} G_\nu( z)$$
$$(\text{resp. }\ \gamma_{r-j} < 1/\lim_{z\uparrow a} G_\nu( z))$$
being equivalent to the existence of elements of $O$ 
greater (resp. lower) than the maximum (resp. minimum) of the support of $\nu$, 
and the limiting points $G_{\nu}^{-1}(1/\gamma_i)$, (resp. $G_{\nu}^{-1}(1/\gamma_{r-j})$) 
being solutions of $F_\nu(\rho )=\gamma_i$, (resp. $F_\nu(\rho )=\gamma_{r-j}$).
Notice that Theorem \ref{mainresult} deals in addition with the outliers of $X_N$ 
located in bounded components of the complement of the support of $\nu $.
\end{exemple}

\begin{exemple}
When $B_N$ is drawn from the GUE($N$,$\sigma ^2$), 
which satisfies the assumptions stated in the preceding remark 
(invariance by unitary conjugation, almost sure weak convergence of the spectral measure 
towards the semicircular distribution \cite{Arnold}, almost sure convergence to $0$ of 
the distance between the eigenvalues and the semicircular support \cite{BaiYin}), 
one recovers a particular case of the results on the outliers from \cite{CDFF10}. 
Indeed, in this setting, $\nu$ is the semicircular distribution, 
which is $\boxplus $-infinitely divisible, and the set $O$ is therefore 
described by Remark \ref{infinidiv}. The conclusion of Theorem \ref{mainresult} 
is then exactly the one of the main result of \cite{CDFF10}. 

Analogously, when $B_N$ is a $N\times N$ Wishart matrix, 
which also satisfies the assumptions stated in the preceding remark 
(invariance by unitary conjugation, almost sure weak convergence of the spectral measure 
towards the Marchenko-Pastur distribution, almost sure convergence to $0$ of 
the distance between the eigenvalues and the Marchenko-Pastur support \cite{Y}), 
one recovers the result on the outliers established by the two last authors of this paper and presented in \cite{MAXIME}. 
In this setting, $\nu$ is the Marchenko-Pastur distribution, 
which is $\boxplus $-infinitely divisible, and the set $O$ is therefore 
described by Remark \ref{infinidiv}. The conclusion of Theorem \ref{mainresult} 
is then exactly the one of the main result of chapter 7 of \cite{MAXIME}. 
\end{exemple}

\subsection{Reduction of the problem to the almost sure convergence of a $r\times r$ matrix}

In this section, we explain how we reduce the problem of locating outliers of $X_N$ 
to a problem of convergence of a certain $r\times r$ matrix in the spirit of \cite{BGRao09}.
Due to the invariance of the Haar measure under multiplication by any unitary matrix, 
we may assume without loss of generality that both $A_N$ and $B_N$ are real diagonal matrices: 
$$A_N=\displaystyle{\text{ Diag}(\underbrace{\theta _1, \ldots , \theta _1}_{k_1}, \ldots , 
\underbrace{\theta _J, \ldots , \theta _J}_{k_J}, \alpha _1^{(N)}, \ldots , \alpha _{N-r}^{(N)})},$$
$$B_N=\text{ Diag}(\beta _1^{(N)}, \ldots , \beta _N^{(N)}).$$
Moreover, from the beginning of our argument, we will make use of the following additive decomposition of $A_N$: 
$$A_N=A_N'+A_N'',$$
$$A_N'=\text{ Diag}(\alpha , \ldots , \alpha , \alpha _1^{(N)}, \ldots , \alpha _{N-r}^{(N)}),$$
$$A_N''=\displaystyle{\text{ Diag}(\underbrace{\theta _1-\alpha , \ldots , \theta _1-\alpha }_{k_1}, \ldots , 
\underbrace{\theta _J-\alpha , \ldots , \theta _J-\alpha }_{k_J}, 0, \ldots , 0)},$$
where the choice of $\alpha \in \mathrm{supp}(\mu )$ is made so that 
$\lim_{y\downarrow0}G_\mu(iy+\alpha)\in\mathbb R+i[-\infty,0)$. 

\noindent
Note that $A_N''={}^tP\Theta P$, where $P$ is the $r \times N$ matrix defined by
$$P=(I_r \vert 0_{r\times (N-r)}),$$
$\Theta$ is the $r\times r$ matrix 
$$\Theta =\displaystyle{\text{ Diag}(\underbrace{\theta _1-\alpha , \ldots , \theta _1-\alpha }_{k_1}, \ldots , 
\underbrace{\theta _J-\alpha , \ldots , \theta _J-\alpha }_{k_J})},$$
and ${}^tX$ denotes the transpose of the matrix $X$.
Under our assumptions, the spectral measure of $A_N'$ (resp. $B_N$) weakly converges to $\mu $ (resp. $\nu $), 
and all eigenvalues of $A_N'$ (resp. $B_N$) belong to any given neighborhood of $\mathrm{supp}(\mu )$ 
(resp. $\mathrm{supp}(\nu )$) for large enough $N\in \mathbb{N}$. Hence, applying Corollary 3.1 of \cite{ColMal11}, 
one gets that, for any $k\in \mathbb{N}^*$, almost surely, 
$$\exists N_k\in \mathbb{N}, \forall N\geq N_k, 
\mathrm{sp}(A_N'+U_N^*B_NU_N)\subseteq K_{\frac{1}{k}}^\mathbb{R}.$$ 
We obtain thus the almost sure existence of a sequence $(\eta _N)_{N\geq N_1}$ 
of positive numbers converging to $0$ so that: $$\forall N\geq N_1, 
\mathrm{sp}(A_N'+U_N^*B_NU_N)\subseteq K_{\eta _N}^\mathbb{R}.$$
(Choose for instance $\eta _N=\frac{1}{k}$, for $N_k\leq N < N_{k+1}$.)
In the following, we will restrict ourselves to the almost sure event on which 
the sequence $(\eta _N)_{N\geq N_1}$ is well-defined. Then, 
for $N\geq N_1$, for any $\lambda \in \mathbb{C}\setminus K_{\eta _N}^\mathbb{R}$, 
the matrix $\lambda I_N-(A_N'+U_N^*B_NU_N)$ is invertible and 
$$\det (\lambda I_N-X_N)=\det (\lambda I_N-(A_N'+U_N^*B_NU_N))\det (I_N-R_N(\lambda ){}^tP\Theta P),$$
where 
\begin{equation}\label{RN}
R_N(\lambda )=\left(\lambda I_N-(A_N'+U_N^*B_NU_N)\right)^{-1}.
\end{equation}
Using that, for rectangular matrices $X\in M_{N,r}(\mathbb{C}), Y\in M_{r,N}(\mathbb{C})$, 
one has $\det (I_N-XY)=\det (I_r-YX)$, one obtains: 
$$\det (\lambda I_N-X_N))=\det (\lambda I_N-(A_N'+U_N^*B_NU_N))\det (I_r-PR_N(\lambda ){}^tP\Theta ).$$
Hence, for $N\geq N_1$, the eigenvalues of $X_N$ outside $K_{\eta _N}^\mathbb{R}$ 
are precisely the zeros of $\det (I_r-PR_N(\lambda ){}^tP\Theta )$ in that open set. 
We will denote by 
\begin{equation}\label{MN}
M_N:=I_r-PR_N{}^tP\Theta 
\end{equation}
the analytic function defined on $\mathbb{C}\setminus K_{\eta _N}^\mathbb{R}$, 
with values in the set of $r\times r$ complex matrices.\\

Then, the following fundamental lemma allows to reduce the problem to  the convergence of the sequence of analytic functions $(M_N)_{N\geq N_1}$. 

\

\begin{lemme}\label{alt-Benaych-Rao}
Let $M\colon\overline{\mathbb{C}}\setminus K\to M_r(\mathbb C)$
be a normal-operator-valued analytic function 
(i.e. $M(z)\in M_r(\mathbb C)$ is normal for each $z\in\overline{\mathbb{C}}\setminus K$) so that 
\begin{enumerate}
\item[(a)]  $\forall z\in \mathbb{C}\setminus K, M(z)^*=M(\overline{z})$. 
\item[(b)] $\Im z>0\implies \Im M(z)$ invertible. 
\end{enumerate}
Assume that there exists a sequence of positive numbers $\{\eta_N\}_{N\in\mathbb N}$ 
decreasing to zero and a sequence of analytic maps
$M_N\colon\overline{\mathbb{C}}\setminus K_{\eta_N}^{\mathbb{R}}\to M_r(\mathbb{C})$ so that
\begin{enumerate}
\item there exists $C>0$ such that for all $ z \in \mathbb{C} $ such that $\vert z \vert >C$, for any $N$, $M_N(z)$ is invertible,
\item for any $z\in\mathbb{C}\setminus\mathbb{R}$ we have $M_N(z)\in GL_r(\mathbb{C})$.
\item for any $\eta >0$, $M_N$ converges to $M$, 
uniformly on $\overline{\mathbb{C}}\setminus K_{\eta }^{\mathbb{C}}$;
\end{enumerate}
If, for a fixed $\eta >0$ such that the boundary points of $K_\eta^{\mathbb{R}}$ are not zeroes of $\det(M)$, $\{\rho_1,\dots,\rho_{p(\eta)}\}$ is the set 
of points $z\in\overline{\mathbb{C}}\setminus K^\mathbb R_\eta$ such 
that $M(z)$ is not invertible, then 
\begin{enumerate}
\item[(i)] $\{\rho_1,\dots,\rho_{p(\eta)}\}\subset\mathbb{R}$;
\item[(ii)] $\mathrm{dim}(\mathrm{Ker}(M(\rho_j))$ equals the order of zero of 
$\det(M(\rho_j))$;
\item[(iii)] For any $0 < \varepsilon < \frac{1}{2}\min\{|\rho_i-\rho_j|,  ~d(\rho_i, K_\eta^{\mathbb{R}})\colon
1\leq i\neq j\leq
{p(\eta)}\}$, there exists an $N_0\in\mathbb{N}$
so that for any $N\geq N_0$ the function $\det(M_N)$ is defined on 
$\overline{\mathbb{C}}\setminus K_{\eta }^{\mathbb{R}}$,  has
exactly $\mathrm{dim}(\mathrm{Ker}(M(\rho_j)))$ zeroes in $ (\rho_j-\varepsilon,\rho_j+\varepsilon)$ for any $j\in \{1,\ldots,\rho_{p(\eta)}\}$
and exactly 
 $\mathrm{dim}(\mathrm{Ker}(M(\rho_1)))+\cdots+\mathrm{dim}(\mathrm{Ker}(M(\rho_{p(\eta)})))$
zeros in $\overline{\mathbb{C}}\setminus K_{\eta }^{\mathbb{R}}$, counted
with multiplicity, so that  $$\{z\in\overline{\mathbb{C}}\setminus K_{\eta }^{\mathbb{R}}\colon \det(M_N(z))=0\}
\subset\bigcup_{j=1}^{p(\eta)}(\rho_j-\varepsilon,\rho_j+\varepsilon).$$
\end{enumerate} 
\end{lemme}

\noindent{\bf Proof:} To begin with, by \cite{GloVid73}, 
$M(z)M(z')=M(z')M(z)$ for all $z, z'\in \overline{\mathbb{C}}\setminus K$.
Thus, there exists a unitary matrix $U$ so that for all $z\in \overline{\mathbb{C}}\setminus K$
$$U^*M(z)U=\mathrm{diag}(h_1(z), \dots , h_r(z)).$$
 Pick an $x\in\mathbb{R}\setminus K$ so that $\det M(x)=0$. 
The condition $M(z)^*=M(\overline{z})$ implies that $M(x)^*=M(x)$ whenever $x\in\mathbb R$, and thus $h_i(x)\in \mathbb{R}$ for any $i\in \{1,\ldots,r\}$. 
Let $I=\{i_1,\ldots, i_j\}$ be such that $h_{l}(x)=0$ if $l\in I$ and $h_l(x)\neq 0$ else.
Denote by $m_k$ the multiplicity of the zero of $h_{i_k}(z)$ at $x$. 
The zero of $\det(M(x))$ is of order equal to $m_1+m_2+\cdots +m_j$.
We only need now to argue that $m_k=1$ for all $1\le k\le j$.
According to (b),  $\Im h_j(z)\neq 0$ whenever $\Im z\neq 0$. This, in particular,
implies that $\Im z\Im h_j(z)$ has constant sign on half-planes, and
by the Julia-Carath\'{e}odory Theorem, $h_j'(x)\neq0$ for any
$x\in\mathbb{R}\setminus K$. This proves (i) and (ii).\\
If we denote by $f(z)=\det(M(z))$ and $f_N(z)=\det(M_N(z)),$ 
then Hurwitz's Theorem \cite[Kapitel 8.5]{Remmert84} guarantees that for $N$ large enough,
$f_N$ will have exactly as many zeros - multiplicity included - as
$f$ has in $\overline{\mathbb{C}}\setminus K_{\eta }^{\mathbb{C}}$ - 
and since all zeros of $f_N$ are known to be real, in 
$\overline{\mathbb{C}}\setminus K_{\eta }^{\mathbb{R}}$ -
and these zeros will cluster towards $\{\rho_1,\dots,\rho_{p(\eta)}\}$
in the sense that for any given $\varepsilon>0$ there exists an 
$N_\varepsilon\in\mathbb N$ so that 
$$\{z\in\overline{\mathbb{C}}\setminus K_{\eta }^{\mathbb{R}}\colon \det(M_N(z))=0\}
\subset \bigcup_{j=1}^{p(\eta)}B(\rho_j,\varepsilon)$$
whenever $N\ge N_\varepsilon$. Moreover, for $\varepsilon>0$ small 
enough, there are exactly $\mathrm{dim}(\mathrm{Ker}(M(\rho_j))$ zeros of $f_N$
in $B(\rho_j,\varepsilon)$, multiplicity included.
Since by 2., $M_N$ is invertible in the two half-planes, we must have 
$$\{z\in \overline{\mathbb{C}}\setminus K_{\eta }^{\mathbb{R}}\colon \det(M_N(z))=0\}
\subset\bigcup_{j=1}^{p(\eta)}(\rho_j-\varepsilon,\rho_j+\varepsilon).$$
$\Box$\\

\section{Convergence of $M_N$ defined by (\ref{MN})}

\subsection{Preliminary results on the resolvent $R_N$}

We begin this section by recording some facts on the resolvent $R_N$ defined by \eqref{RN}.
Recall that, if $X$ is a selfadjoint operator on a Hilbert space with spectrum $\sigma (X)$, 
then we shall denote by $R_X(z)=(z-X)^{-1}$ its resolvent. 
It is known that this resolvent is analytic on $\mathbb{C}\setminus \sigma (X)$. 
If $\varphi $ is a positive unital linear functional on the unital algebra generated 
by $X$, then the distribution $\mu _{X,\varphi }$ of $X$ with respect to $\varphi $ can be 
recovered as $$G_{\mu _{X,\varphi }}(z)=\varphi (R_X(z)),\quad z\notin \sigma (X).$$
Since in most cases it will be clear from the context which 
functional $\varphi $ is considered, we shall suppress $\varphi $
from the notation $\mu _{X,\varphi }$.\\
A more general notion of resolvent of $X$, which we shall use only
sparingly in this paper, can be defined, following Voiculescu,
as below: for an arbitrary operator $b$ on the same 
Hilbert space as $X$, we can write its decomposition
$$b=\underbrace{\frac{b+b^*}{2}}_{\Re b}+i\underbrace{\frac{b-b^*}{2i}}_{\Im b},$$
where $\Re b,\Im b$ are selfadjoint. We shall write $\Im b > 0$ if 
$\Im b\geq 0$ as operator on Hilbert space, and $(\Im b)^{-1}$ exists
and is bounded. It has been noted by Voiculescu \cite{Voiculescu00} that 
\begin{equation}\label{Voi}
R_X(b)=(b-X)^{-1}, \quad \Im b>0
\end{equation}
is an analytic map so that $\Im R_X(b)<0$. Moreover, as noted
in \cite[Remark 2.5]{BPV12}, if $E$ is a positive unit-preserving linear map
which leaves the algebra of $b$ invariant, then 
\begin{equation}\label{BPV}
\Im \left[E[R_X(b)]\right]^{-1}\geq\Im b.
\end{equation}

\

\noindent
For all $z\in \mathbb{C}\setminus \mathbb{R}$, $R_N(z)$ defined by \eqref{RN} satisfies: 
\begin{equation}\label{partieimaginaire}
\forall z\in \mathbb{C}\setminus \mathbb{R}, \Vert R_N(z)\Vert \leq \frac{1}{|\Im z|}.
\end{equation}
This implies that, for any $z\in \mathbb{C}\setminus \mathbb{R}$, 
the matrix $\mathbb{E}[R_N(z)]$ is well-defined and also satisfies: 
$$\Vert \mathbb{E}[R_N(z)]\Vert \leq \frac{1}{|\Im z|}.$$

\begin{lemme}\label{l}
When $\Im z\neq0$, $\mathbb{E}[R_N(z)]$ is an invertible matrix and 
\begin{equation} \label{borne} 
\forall z\in \mathbb{C}\setminus \mathbb{R},\  
\Vert \mathbb{E}[R_N(z)]^{-1}\Vert < |z|+C_1+\frac{4C_2}{|\Im z|}, 
\end{equation}
where $C_1$ is any constant greater than $\sup_N (\Vert A^{'}_N \Vert + \Vert B_N \Vert)$ 
and $C_2$ any constant greater than $\sup_N \left(\tr_N(B_N^2)-[\tr_N(B_N)]^2\right)$. 
\end{lemme}

\begin{proof}
As noted in equation \eqref{Voi} above applied to $b=z-A_N'$,
$\Im (z-A_N'-U_N^*B_NU_N)^{-1}<0$. Since $\mathbb E$ is both positive
and faithful, it follows that for any $z\in\mathbb{C}^+$, 
$\Im\mathbb E\left[(z-A_N'-U_N^*B_NU_N)^{-1}\right]<0$, and thus by the same remark of Voiculescu \cite{Voiculescu00}, $\mathbb E\left[(z-A_N'-U_N^*B_NU_N)^{-1}\right]$ is invertible.
The second statement of the lemma is equivalent to a statement about the
power series expansion of $z\mapsto\mathbb{E}[R_N(z)]^{-1}$ around
infinity. The power series expansion 
$$\mathbb{E}[R_N(z)]=\sum_{n=0}^\infty\frac{\mathbb{E}[(A_N^{'}+U_N^*B_NU_N)^n]}{z^{n+1}},
\quad|z|>\|A_N^{'}\|+\|B_N\|,$$
assures us that
\begin{eqnarray*}
f(z) & := & (z-A_N^{'}-\mathbb{E}[U_N^*B_NU_N])-\mathbb{E}[R_N(z)]^{-1}\\
& = & z-\mathbb{E}[A_N^{'}+U_N^*B_NU_N]-\left[\sum_{n=0}^\infty\frac{\mathbb{E}[(A_N^{'}+U_N^*B_NU_N)^n]}{z^{n+1}}\right]^{-1}\\
& = & \left\{\left(z-\mathbb{E}[A_N^{'}+U_N^*B_NU_N]\right)\left[\sum_{n=0}^\infty\frac{\mathbb{E}[(A_N^{'}+U_N^*B_NU_N)^n]}{z^{n+1}}\right]-1\right\}\\
& & \mbox{}\times\left[\sum_{n=0}^\infty\frac{\mathbb{E}[(A_N^{'}+U_N^*B_NU_N)^n]}{z^{n+1}}\right]^{-1}\\
& = & \frac1z\left[\mathbb E[(A_N^{'}+U_N^*B_NU_N)^2]-
\mathbb E[A_N^{'}+U_N^*B_NU_N]^2+\frac1z\cdot O(1)\right]\\
& & \mbox{}\times\left[1+\sum_{n=1}^\infty\frac{\mathbb{E}[(A_N^{'}+U_N^*B_NU_N)^n]}{z^{n}}\right]^{-1}\\
& = & \frac1z\left[\mathbb E[(A_N^{'}+U_N^*B_NU_N)^2]-
\mathbb E[A_N^{'}+U_N^*B_NU_N]^2\right]+\frac1{z^2}O(1).
\end{eqnarray*}
This power series expansion holds uniformly in $N$ as long as
$\|A_N^{'}\|+\|B_N\|$ is bounded uniformly in $N$. In particular,
we obtain 
$$
\lim_{z\to\infty}zf(z)=\mathbb E[(A_N^{'}+U_N^*B_NU_N)^2]-
\mathbb E[A_N^{'}+U_N^*B_NU_N]^2=\tr_N((B-\tr_N(B))^2)\cdot1,
$$
uniform limit in $N$.

As noted in equation \eqref{BPV} above, $\Im f(z)<0$, so for any positive linear functional 
$\varphi$ on $M_N(\mathbb{C})$, the function $z\mapsto
\varphi(f(z))$ maps $\mathbb C^+$ into the lower half-plane, 
and $\lim_{z\to\infty}z\varphi(f(z))=\varphi(1)(\tr_N(B^2)-[\tr_N(B)]^2).$
Thus, $z\mapsto \varphi(f(z))$ is the Cauchy-Stieltjes transform
of a positive measure supported on $[-\|A_N^{'}\|-\|B_N\|,\|A_N^{'}\|+\|B_N\|]$
of total mass $\varphi(1)(\tr_N(B^2)-[\tr_N(B)]^2).$ It follows that
$$
|\varphi(f(z))|<\frac{\varphi(1)}{\Im z}(\tr_N(B^2)-[\tr_N(B)]^2),\quad z\in
\mathbb C^+.
$$
Now, since positive linear functionals on von Neumann algebras 
reach their norm on the unit, the Jordan decomposition of 
linear functionals allows us to write
$$
\|f(z)\|\leq\sup_{\|\varphi\|=1}|\varphi(f(z))|\leq4
\sup_{\varphi\ge0,\varphi(1)=1}|\varphi(f(z))|<\frac{4}{\Im z}
(\tr_N(B^2)-[\tr_N(B)]^2).
$$
Since $\mathbb E[R_N(z)]^{-1}=z-\mathbb E[A_N^{'}+U_N^*B_NU_N]-f(z)$, for
any $z\in\mathbb C\setminus\mathbb R$, we conclude that
$$
\|\mathbb E[R_N(z)]^{-1}\|\leq|z|+\|A_N^{'}\|+\|B_N\|+\|f(z)\|<
|z|+C_1+\frac{4C_2}{|\Im z|},
$$
as stated in our lemma.
\end{proof}

\

\noindent
One now states concentration results that will allow to reduce 
the proof of Proposition \ref{uniformconvergence} to the convergence 
of a sequence of deterministic matrices and to estimate the variance of each entry of the resolvant $R_N(z)$. 

\begin{lemme}\label{concentration}
\begin{description}
\item[(i)] $\forall z\in \mathbb{C}\setminus \mathbb{R}, PR_N(z)^t P- P\mathbb{E}[R_N(z)]^t P
\underset{N\rightarrow +\infty}{\overset{a.s.}{\longrightarrow }}0.$
\item[(ii)] $\forall z\in \mathbb{C}\setminus \mathbb{R}, \forall (k,l) \in \{1, \ldots , N\}^2, 
\mathbb{V}((R_N(z))_{kl})\leq \frac{C}{N|\Im z|^4}.$
\end{description}
\end{lemme}

\noindent {\bf Proof.} 
Fix $z\in \mathbb{C}\setminus \mathbb{R}$. 
It is clear from the remark that writing $P(R_N(z)-\mathbb{E}[R_N(z)]){}^tP$ 
corresponds to taking the upper left $r\times r$ corner of $R_N(z)-\mathbb{E}[R_N(z)]$, 
that $(i)$ is equivalent to: 
\begin{equation} \label{concentration1} 
\forall (k,l) \in \{1, \ldots , r\}^2, 
\left(R_N(z)-\mathbb{E}[R_N(z)]\right)_{kl}\underset{N\rightarrow +\infty}{\overset{a.s.}{\longrightarrow }}0.
\end{equation}
Now, for any $(k,l) \in \{1, \ldots , N\}^2$, since the function $$f: U_N\mapsto R_N(z)_{kl}$$ 
is Lipschitz on the unitary group $\mathbb{U}_N$ with Lipschitz bound $\frac{C}{\vert \Im z \vert^2}$,
by Corollary 4.4.28 of the book \cite{AGZ10}, for any $0 < \alpha < \frac{1}{2}$,
$$\mathbb{P}\left(\vert\left(R_N(z)-\mathbb{E}[R_N(z)]\right)_{kl}\vert >  \frac{\epsilon}{N^{\frac{1}{2}-\alpha}}\right)\leq 2 \exp \left(-C N^{2\alpha} \vert \Im z \vert^4 \epsilon^2 \right).$$
Hence, one gets $\eqref{concentration1}$ by a standard application of Borel-Cantelli lemma, 
and $(ii)$ by the classical formula holding for a positive random variable $X$:  $$\mathbb{E}(X)=\int_0^{+\infty}\mathbb{P}(X>t)dt. \quad \Box $$

\subsection{Convergence of $M_N$} 
We investigate the convergence of the sequence of analytic functions $(M_N)_{N\geq N_1}$ defined by (\ref{MN}). 
We need the following preliminary lemma.

\begin{lemme} \label{analyticextension}
The function $\chi: z\mapsto \frac{1}{\omega (z)-\alpha }$ is analytic on 
$\mathbb{C}\setminus \mathrm{supp}(\mu \boxplus \nu )$ and satisfies 
$\chi(\overline{z})=\overline{\chi(z)}$ for any $z$ in $\mathbb{C}\setminus \mathrm{supp}(\mu \boxplus \nu )$.
\end{lemme}

\noindent {\bf Proof.} 
This lemma readily follows  from Lemma \ref{extension}. 
$\Box$

\begin{proposition} \label{uniformconvergence}
Almost surely, for any $\eta > 0$, the sequence $(M_N)_{N\geq N_0}$, where $N_0\geq N_1$ is such that $\forall N \geq N_0,~\eta_N < \eta$,  converges to 
$$M = \mathrm{diag}(1 - (\theta_1-\alpha) \chi, \ldots , 1 - (\theta_J-\alpha) \chi),$$ 
uniformly on compact subsets of $\mathbb{C}\setminus K_\eta ^\mathbb{C}$.
\end{proposition}

\noindent
Since $$M_N=I_r-PR_N{}^tP\Theta ,$$ it is equivalent to prove the convergence of 
$PR_N{}^tP$ towards $\chi I_r$.

\

\noindent
We first study the convergence of the sequence of analytic functions $(\mathbb{E}[R_N])_{N\geq 1}$: 

\begin{proposition} \label{estimationinmean}
\begin{equation} \label{secondapproximation}
\forall z\in \mathbb{C}\setminus \mathbb{R}, ~~P\mathbb{E}[R_N(z)]{}^tP 
\underset{N\rightarrow +\infty}{\longrightarrow }\chi(z)I_r.
\end{equation}
\end{proposition}

\noindent 
Fix $z\in \mathbb{C}^+$ (which is sufficient by a reflection argument). 
We now break the proof in three lemmas. 
First, a strenghtening of the result from \cite{Kargin11} stating
that that the matrix 
$\mathbb{E}[R_N(z)]$ is diagonal:
\begin{lemme}\label{bicommutant}
If $b\in M_N(\mathbb C)$ is so that $b-U_N^*B_NU_N$ is invertible
for each value of the Haar unitary $U_N$, then $\mathbb E\left[(b-U_N^*
B_NU_N)^{-1}\right]\in\{b\}''$, where the bicommutant $\{b\}''$
is taken in $M_N(\mathbb C)$.
\end{lemme}

\noindent {\bf Proof.}
Pick an arbitrary unitary $V$ in the commutant $\{b\}'$ of $b$.
By the invariance of the Haar measure, 
\begin{eqnarray*}
V^*\mathbb E\left[(b-U_N^*
B_NU_N)^{-1}\right]V & = & \mathbb E\left[(V^*bV-V^*U_N^*
B_NU_NV)^{-1}\right]\\
& = &
\mathbb E\left[(b-U_N^*B_NU_N)^{-1}\right],
\end{eqnarray*}
so that $\mathbb E\left[(b-U_N^*
B_NU_N)^{-1}\right]\in\{V\}'$. Thus, since
a von Neumann algebra equals the span of its unitaries,
$\mathbb E\left[(b-U_N^*B_NU_N)^{-1}\right]\in\{b\}''$, 
as claimed. $\Box$

\

\noindent
Recall that the first $r$ eigenvalues of $A_N'$ are all equal to
$\alpha$. We apply the above lemma to $b=z-A_N'$ to conclude that 
the first $r$ eigenvalues of $\mathbb E\left[R_N(z)\right]$ are all
equal, and thus $~~P\mathbb{E}[R_N(z)]{}^tP=\chi_N(z)I_r$ for the
Cauchy-Stieltjes transform $\chi_N$ of some probability measure 
depending on $A_N',B_N$.

\

\noindent
Our next task is to establish an approximate matricial subordination equation, 
namely to prove that $\mathbb{E}[R_N(z)]$ is asymptotically 
equal to $(\omega _N(z)I_N-A_N^{'})^{-1}$, for a certain complex number $\omega _N(z)$. 
Then, we prove the uniform convergence on the compact subsets of $\mathbb{C}^+$
of the sequence of analytic functions $(\omega _N)_{N\geq 1}$ towards $\omega $. 

\begin{lemme} \label{approximatesubordination}
For $z\in \mathbb{C}^+$, one has: 
$$\Vert \mathbb{E}[R_N(z)]-(\omega _N(z)I_N-A_N')^{-1}\Vert \underset{N\rightarrow +\infty}{\longrightarrow }0,$$
where 
\begin{equation}\label{omegaN} 
\omega _N(z):=\frac{1}{\mathbb{E}[R_N(z)]_{11}}+ \alpha=\frac{1}{\chi_N(z)}+\alpha .
\end{equation}
\end{lemme}

\noindent {\bf Proof.}
First notice, using \eqref{BPV}, that $\omega_N$ defined by \eqref{omegaN}, satisfies:
\begin{equation}\label{ImomegaN}
\forall z\in \mathbb{C}^+,\ \Im \omega_N(z)\geq \Im z.
\end{equation}
Fix $z\in \mathbb{C}^+$ and define $$\Omega _N(z):=\mathbb{E}[R_N(z)]^{-1}+A_N',$$
which belongs, according to Lemma \ref{bicommutant}, to $\{A_N'\}''$.
We denote its $k$-th diagonal entry $$(\Omega _N(z))_{kk}=\frac{1}{\mathbb{E}[R_N(z)]_{kk}}+ (A_N')_{kk}.$$ 
Note that $\omega _N(z)=(\Omega _N(z))_{11}$ and that, 
using \eqref{partieimaginaire} and \eqref{ImomegaN}, 
$$\Vert (\omega _N(z)-A_N')^{-1}\Vert \leq \frac{1}{\Im\omega_N(z)}\leq\frac{1}{\Im z}$$
\noindent and
$$\Vert (\Omega _N(z)-A_N')^{-1}\Vert \leq \Vert E[R_N(z)]\Vert \leq \frac{1}{\Im z}.$$
If we prove that: 
$$\exists C'(z) > 0, \forall N\geq 1, \forall (k,l)\in \{1, \ldots , N\}, 
\vert (\Omega _N(z))_{kk}-(\Omega _N(z))_{ll}\vert \leq \frac{C'(z)}{N},$$
then we may conclude:\\

$\Vert \mathbb{E}[R_N(z)]-(\omega _N(z)I_N-A_N')^{-1}\Vert$
\begin{eqnarray*}
 &=&\Vert (\Omega _N(z)-A_N')^{-1}-(\omega _N(z)I_N-A_N')^{-1}\Vert \\
&=&\Vert (\Omega _N(z)-A_N')^{-1}(\omega _N(z)I_N-\Omega _N(z))(\omega _N(z)I_N-A_N')^{-1}\Vert \\
&\leq &\Vert (\Omega _N(z)-A_N')^{-1}\Vert \Vert (\omega _N(z)I_N-\Omega _N(z))\Vert \Vert (\omega _N(z)I_N-A_N')^{-1}\Vert \\
&\leq &\frac{1}{|\Im z|^2}\Vert (\Omega _N(z))_{11}I_N-\Omega _N(z)\Vert \\
&\leq &\frac{C'(z)}{N|\Im z|^2}\\
&\underset{N\rightarrow +\infty}{\longrightarrow }&0\\
\end{eqnarray*}
For $(k,l)\in \{1, \ldots , N\}$, observe that 
\begin{equation}\label{Omega}(\Omega _N(z))_{kk}-(\Omega _N(z))_{ll}=(\Omega _N(z)E_{kl}-E_{kl}\Omega _N(z))_{kl}.\end{equation}
Define, for a given deterministic matrix $X\in M_N(\mathbb{C})$, 
\begin{eqnarray*}
\Delta_X &:=& \mathbb{E}[(R_N(z)-\mathbb{E}[R_N(z)])(A_N'X -XA_N')(R_N(z)-\mathbb{E}[R_N(z)])]\\
&=&\mathbb{E}[R_N(z)(A_N'X -XA_N')R_N(z)]-\mathbb{E}[R_N(z)](A_N'X -XA_N')\mathbb{E}[R_N(z)]\\ 
\end{eqnarray*}
Noting that for any deterministic Hermitian matrix $X$, $$\frac{{d}}{dt}  \mathbb{E}\left( (z-A_N' -e^{-itX} U_N^* B_N U_Ne^{itX})^{-1} \right)_{|_{t=0}}=0,$$
we readily deduce that 
$$\mathbb{E}[R_N(z)(A_N'X -XA_N')R_N(z)]=\mathbb{E}[R_N(z)]X-X\mathbb{E}[R_N(z)],$$
and then extend this identity by linearity to any matrix $X\in M_N(\mathbb{C})$.
It follows that 
\begin{equation} \label{difference}
\Omega _N(z)X-X\Omega _N(z) = -\mathbb{E}[R_N(z)]^{-1}\Delta _X\mathbb{E}[R_N(z)]^{-1}.
\end{equation}
For $X=E_{kl}$, (\ref{Omega}) and (\ref{difference}) yield\\

\noindent $\vert (\Omega _N(z))_{kk}-(\Omega _N(z))_{ll}\vert$
\begin{eqnarray*}
 &=& \vert (\mathbb{E}[R_N(z)]^{-1}\Delta _{E_{kl}}\mathbb{E}[R_N(z)]^{-1})_{kl}\vert \\
&\leq &\vert \mathbb{E}[R_N(z)]^{-1}_{kk}\vert \vert (\Delta _{E_{kl}})_{kl}\vert \vert \mathbb{E}[R_N(z)]^{-1}_{ll}\vert \\
&\leq &\left(|z|+C_1+\frac{4C_2}{|\Im z|}\right)^2{\vert (A_N')_{kk}-(A_N')_{ll}\vert }\\
& & \mbox{}\times\mathbb{E}[\vert (R_N(z)-\mathbb{E}[R_N(z)])_{kk}(R_N(z)-\mathbb{E}[R_N(z)])_{ll}\vert ]\\
&\leq &{2\Vert A_N'\Vert } \left(|z|+C_1+\frac{4C_2}{|\Im z|}\right)^2\mathbb{V}((R_N(z))_{kk})^{\frac{1}{2}}\mathbb{V}((R_N(z))_{ll})^{\frac{1}{2}}\\
&\leq &\frac{C\left(\vert z\vert + 1 + \frac{1}{|\Im z|}\right)^2 }{N|\Im z|^4}\\
\end{eqnarray*}
where we used (\ref{borne}) and Lemma \ref{concentration} (ii) in the three last inequalities.
And we are done. $\Box$\\

\noindent 
We now study the convergence of the sequence $(\omega _N)_{N\geq 1}$ defined by (\ref{omegaN}).

\begin{lemme} \label{scalarconvergence}
The sequence of analytic functions $(\omega _N)_{N\geq 1}$ defined on 
$\mathbb{C}\setminus \mathbb{R}$ converges uniformly towards $\omega $ 
on the compact subsets of $\mathbb{C}\setminus \mathbb{R}$. 
\end{lemme}

\noindent {\bf Proof.}
It follows from Lemma \ref{approximatesubordination}, by taking the normalized trace, and using the notation 
$$g_N(z):=\mathbb{E}[G_{\mu_{A_N'+U_N^*B_NU_N}}(z)]$$ that 
$$g_N(z)-G_{\mu _{A_N'}}(\omega _N(z))\underset{N\rightarrow +\infty}{\longrightarrow }0.$$
Using Lemma 7.7 in \cite{Capitaine11} and \eqref{partieimaginaire}, we deduce that 
\begin{equation} \label{equationapprochee} 
g_N(z)-G_{\nu}(\omega _N(z))\underset{N\rightarrow +\infty}{\longrightarrow }0.
\end{equation}
The sequence of analytic functions $(\omega _N)_{N\geq 1}$ is normal, 
and thus there exists at least one converging subsequence. 
For any fixed $z\in \mathbb{C}^+$, let us consider a converging subsequence 
$\omega _{\phi (N)}(z)$ of $\omega _N(z)$ and denote by $l(z)$ the limit. 
As noted above in \eqref{ImomegaN}, $\Im \omega _N(z)\ge\Im z$.  
Thus, if $\Im z$ is large enough, $|\omega _N(z)|$ will be large, 
and we can then uniquely invert with respect to composition uniformly 
in $N$ in \eqref{equationapprochee} to obtain $G_\nu ^{-1}(g_N(z)-o(1))=\omega _N(z)$. 
Letting $N$ go to infinity, Voiculescu's asymptotic freeness result guarantees that 
$g_N(z)\to G_{\mu \boxplus \nu }(z)$. 
Thus, $l(z)=G_\nu ^{-1}(G_{\mu \boxplus \nu }(z))$, independent of the convergent subsequence 
$\omega _{\phi (N)}(z)$ we chose. 
This implies that $\lim_{N\to\infty}\omega _{N}=l$ uniformly on compact sets of 
$\mathbb{C}^+$, and by analytic continuation that $l=\omega $.
Therefore, for any  $z \in \mathbb{C}^+$,  $\omega (z)$ is the  
unique cluster point of  $\omega _N(z)$. $\Box$\\

\noindent {\bf Proof of Proposition \ref{estimationinmean}.}
Fix $z\in \mathbb{C}^+$ (which is sufficient by a reflection argument).  
We proved that
\begin{equation}
P\mathbb{E}[R_N(z)]{}^tP = \frac{1}{\omega _N(z)-\alpha }I_r.
\end{equation}
We simply conclude by using Lemma \ref{scalarconvergence} 
and $\Im \omega_N(z) \geq \Im z$, $\Im \omega(z) \geq \Im z$. $\Box$\\

\noindent 
It will be important in the proof of Proposition \ref{uniformconvergence} 
to have the following analytic continuation result. 

\begin{theoreme} \label{complexanalysislemma}(lemma 7.6.5 \cite{Chatterji})
Let  $\mathcal{A}$ be an open nonempty subset of $\mathbb{C}$ 
and $D\subset \mathcal{A}$ such that $\overline{D}=\overline{\mathcal{A}}$. 
Let $(g_n)$ be a sequence of locally bounded holomorphic functions on 
$\mathcal{A}$ such that $\lim_{n\rightarrow +\infty } g_n(d)$ exists for any $d\in \mathcal{D}$.
Then, the sequence $(g_n)$ converges towards a function $g$ which is holomorphic on 
$\mathcal{A}$, uniformly on each compact subset of $\mathcal{A}$.
\end{theoreme}

\noindent {\bf Proof of Proposition \ref{uniformconvergence}.}
Define $$\mathcal{D}_\eta =\{z \in  \mathbb{C}\setminus K_{\eta}^{\mathbb{C}}, 
\Re z\in \mathbb{Q}, \Im z \in \mathbb{Q}^*\}.$$
According to Lemma \ref{concentration} (i)  and \eqref{secondapproximation}, 
for any $z \in \mathcal{D}_\eta $, almost surely, 
$PR_N(z){}^tP$ converges towards $\chi(z)I_r$.
Hence, almost surely, for any $k,l \in \{1, \ldots , r\}$, 
for any $\eta > 0$, $(R_N)_{kl}$ is a bounded sequence of holomorphic functions 
on $\mathbb{C}\setminus K_{\eta}^{\mathbb{C}}$ such that 
the limit of $(R_N(z))_{kl}$ exists for any $z\in \mathcal{D}_\eta $. 
Therefore, according to Theorem \ref{complexanalysislemma}, 
we can deduce that, almost surely, $(R_N(z))_{kl}$ converges 
towards an holomorphic function $\chi_{kl}$ 
on $\mathbb{C}\setminus K_{\eta}^{\mathbb{C}}$, 
uniformly on each compact subset of $\mathbb{C}\setminus K_{\eta}^{\mathbb{C}}$.
Of course, $\chi_{kl}$ coincides with $\frac{\delta_{kl}}{\omega (z)-\alpha }$ on $\mathbb{C}^+$ 
so that $\chi_{kl}=\delta _{kl}\chi$. The proof is complete. $\Box$\\

\noindent 
We will now prove Theorem \ref{mainresult}, by applying Lemma \ref{alt-Benaych-Rao} 
on an almost sure event on which its assumptions are satisfied. 

\

\noindent {\bf Proof of Theorem \ref{mainresult}.} 
We consider the almost sure event, whose existence is guaranteed by Proposition \ref{uniformconvergence}, 
on which there exists an integer $N_1$, a sequence $(\eta _N)_{N\geq N_1}$ 
of positive numbers converging to $0$, 
so that $$\mathrm{sp}(A_N'+U_N^*B_NU_N)\subseteq K_{\eta _N}^\mathbb{R}$$
and, for any $\eta > 0$, the sequence $(M_N)_{N\geq N_0}$, 
where $N_0\geq N_1$ is such that $\forall N \geq N_0,~\eta_N < \eta$,  converges to 
$$M = \mathrm{diag}(1 - (\theta_1-\alpha)\chi, \ldots , 1 - (\theta_J-\alpha)\chi),$$ 
uniformly on compact subsets of $\mathbb{C}\setminus K_\eta ^\mathbb{C}$. 
On this event, apply Lemma \ref{alt-Benaych-Rao} to the sequence $(M_N)_{N\geq N_0}$ and its uniform limit $M$. 
The funtion $M\colon\overline{\mathbb{C}}\setminus K\to M_r(\mathbb C)$
is indeed a normal-operator-valued analytic function satisfying 
trivially conditions (a) and (b) of Lemma \ref{alt-Benaych-Rao}. 
The sequence $(M_N)_{N\geq N_0}$ consists of (random) analytic maps on
$\overline{\mathbb{C}}\setminus K_{\eta_N}^{\mathbb{R}}$. Condition 3. of Lemma \ref{alt-Benaych-Rao} 
is guaranteed by Proposition \ref{uniformconvergence}. Condition 2. is straightforward. 
To check condition 1., it is sufficient to argue that for any $z$ such that $\vert z \vert >C_1$,
$$\Vert R_N(z) \Vert \leq \frac{1}{d(z,[-C_1,C_1])},$$ 
where $C_1$ is chosen as in Lemma \ref{l}.  
Almost every $\eta >0$ is such that the boundary points of $K_\eta^{\mathbb{R}}$ 
are not zeroes of $\det(M)$, so for such $\eta $'s, we get exactly the conclusion of Theorem \ref{mainresult}. 
Indeed, as explained in Section 3.2, eigenvalues of $X_N$ in $\mathbb{C}\setminus K_{\eta }^{\mathbb{R}}$ 
are exactly zeroes of $\det(M_N)$, and the set of points $z$ such that 
$M(z)$ is not invertible is precisely $O$.
$\Box$\\

\def\cprime{$'$}

$\ $

$\ $

S. T. Belinschi: Queen's University and Institute of Mathematics 

``Simion Stoilow'' of the Romanian Academy.

Address: Department of Mathematics and Statistics, 

Queen's University,
Jeffrey Hall, 

Kingston, ON K7L 3N6, Canada

Email: {sbelinsch@mast.queensu.ca}

$\ $

H. Bercovici: Indiana University.

Address: {Department of Mathematics,}

{ Indiana University, Rawles Hall,}

Bloomington, IN 47405, USA.

Email: {bercovic@indiana.edu}

$\ $

M. Capitaine: CNRS Toulouse.

Address: {CNRS, Institut de Math\'ematiques de Toulouse,}

 Equipe de Statistique et Probabilit\'es,

F-31062 Toulouse Cedex 09, France.

Email: {mireille.capitaine@math.univ-toulouse.fr}

$\ $

M. F\'evrier: Universit\'e Paris Sud.

Address: {Universit\'e Paris Sud,
Laboratoire de Math\'ematiques}, 

B\^{a}t. 425 91405 Orsay Cedex, France.

Email: {maxime.fevrier@math.u-psud.fr}


\begin{thebibliography}{10}

\bibitem{AGZ10}
G.~W. Anderson, A. Guionnet, and O. Zeitouni.
\newblock {\em An introduction to random matrices}, volume 118 of {\em
  Cambridge Studies in Advanced Mathematics}.
\newblock Cambridge University Press, Cambridge, 2010.

\bibitem{Arnold}
L. Arnold. On the asymptotic distribution of the eigenvalues of random matrices.
\newblock {\em J. Math. Anal. Appl.} 20:262–268, 1967.


\bibitem{BaiYao08b}
Z.~D. {Bai} and J.~{Yao}.
On sample eigenvalues in a generalized spiked population model.  
\newblock {\em J. Multivariate Anal.}, doi:10.1016/j.jmva.2011.10.009 

\bibitem{BaiYin}Z.~D. Bai and Y. Q. Yin.
 Necessary and sufficient conditions for almost
sure convergence of the largest eigenvalue of a Wigner matrix.
\newblock {\em Ann. Probab.},  16:
1729–1741, 1988.
\bibitem{BBP05}
J.~Baik, G.~Ben~Arous, and S.~P{\'e}ch{\'e}.
\newblock Phase transition of the largest eigenvalue for nonnull complex sample
  covariance matrices.
\newblock {\em Ann. Probab.}, 33(5):1643--1697, 2005.

\bibitem{BaikSil06}
J.~Baik and J.~W. Silverstein.
\newblock Eigenvalues of large sample covariance matrices of spiked population
  models.
\newblock {\em J. Multivariate Anal.}, 97(6):1382--1408, 2006.

\bibitem{BelBer07}
S.~T. Belinschi and H.~Bercovici.
\newblock A new approach to subordination results in free probability.
\newblock {\em J. Anal. Math.}, 101:357--365, 2007.

\bibitem{BPV12}
S.~T. Belinschi, M.~Popa, and V.~Vinnikov.
\newblock Infinite divisibility and a non-commutative {B}oolean-to-free
  {B}ercovici-{P}ata bijection.
\newblock {\em J. Funct. Anal.}, 262(1):94--123, 2012.

\bibitem{Belinschi08}
S. Belinschi.
\newblock The {L}ebesgue decomposition of the free additive convolution of two
  probability distributions.
\newblock {\em Probab. Theory Related Fields}, 142(1-2):125--150, 2008.

\bibitem{BercoviciVoiculescu}
H. Bercovici and D. Voiculescu.
\newblock Free convolution of measures with unbounded support.
\newblock {\em Indiana Univ. Math. J.}, 42(3):733--773, 1993.

\bibitem{BGRao09}
F.~{Benaych-Georges} and R.~R. {Nadakuditi}.
\newblock {The eigenvalues and eigenvectors of finite, low rank perturbations
  of large random matrices}.
\newblock {\em Advances in Mathematics}, 227(1): 494--521, 2011.

\bibitem{Biane97b}
P. Biane.
\newblock On the free convolution with a semi-circular distribution.
\newblock {\em Indiana Univ. Math. J.}, 46(3):705--718, 1997.

\bibitem{Biane98}
P. Biane.
\newblock Processes with free increments.
\newblock {\em Math. Z.}, 227(1):143--174, 1998.

\bibitem{Capitaine11}
M.~{Capitaine}.
\newblock {Additive/multiplicative free subordination property and limiting
  eigenvectors of spiked additive deformations of Wigner matrices and spiked
  sample covariance matrices}.
\newblock {\em Journal of Theoretical Probability}, 2012, DOI: 10.1007/s10959-012-0416-5.

\bibitem{CDF09}
M.~Capitaine, C.~Donati-Martin, and D.~F{\'e}ral.
\newblock The largest eigenvalues of finite rank deformation of large {W}igner
  matrices: convergence and nonuniversality of the fluctuations.
\newblock {\em Ann. Probab.}, 37(1):1--47, 2009.

\bibitem{CDFF10}
M.~Capitaine, C.~Donati-Martin,  D.~F{\'e}ral and M. F\'evrier.
\newblock Free convolution with a semi-circular distribution and  eigenvalues of spiked deformations of  {W}igner
  matrices. 
 \newblock {\em Electronic Journal of Probability}, 16: 1750--1792, 2011.

\bibitem{Chatterji} S. D. Chatterji.
Cours d'analyse 
2 Analyse complexe
Presses polytechniques et universitaires romandes, 1997.

\bibitem{ColMal11}
B.~{Collins} and C.~{Male}.
\newblock {The strong asymptotic freeness of Haar and deterministic matrices}.
\newblock {\em ArXiv e-prints}, May 2011.

\bibitem{FePe}
D.~F{\'e}ral and S.~P{\'e}ch{\'e}.
\newblock The largest eigenvalue of rank one deformation of large {W}igner
  matrices.
\newblock {\em Comm. Math. Phys.}, 272(1):185--228, 2007.

\bibitem{MAXIME}
M. F\'evrier.
Infinitesimal freeness and deformed matrix models.
Doctorat de l'Universit\'e Paul Sabatier 2010

\bibitem{Fulton98}
W.~Fulton.
\newblock Eigenvalues of sums of {H}ermitian matrices (after {A}. {K}lyachko).
\newblock {\em Ast\'erisque}, (252):Exp.\ No.\ 845, 5, 255--269, 1998.
\newblock S{\'e}minaire Bourbaki. Vol. 1997/98.

\bibitem{FurKom81}
Z.~F{\"u}redi and J.~Koml{\'o}s.
\newblock The eigenvalues of random symmetric matrices.
\newblock {\em Combinatorica}, 1(3):233--241, 1981.

\bibitem{GarnettBook}
John~B. Garnett.
\newblock {\em Bounded analytic functions}, volume~96 of {\em Pure and Applied
  Mathematics}.
\newblock Academic Press Inc. [Harcourt Brace Jovanovich Publishers], New York,
  1981.

\bibitem{GloVid73}
J.~Globevnik and I.~Vidav.
\newblock A note on normal-operator-valued analytic functions.
\newblock {\em Proc. Amer. Math. Soc.}, 37:619--621, 1973.

\bibitem{John}
I. Johnstone.
\newblock On the distribution of the largest eigenvalue in principal components analysis.
\newblock {\em Ann. Stat. }, 29:295--327, 2001.

\bibitem{Kargin11}
V.~{Kargin}.
\newblock {Subordination of the resolvent for a sum of random matrices}.
\newblock {\em ArXiv e-prints}, September 2011.

\bibitem{LV}  
P. Loubaton and P. Vallet. 
Almost sure localization of the eigenvalues in a Gaussian information-plus-noise model. Application to the spiked models 
2010 Available at http://front.math.ucdavis.edu/1009.5807.

\bibitem{Peche06}
S.~P{\'e}ch{\'e}.
\newblock The largest eigenvalue of small rank perturbations of Hermitian
random matrices.
\newblock {\em Probab. Theory Related Fields}, 134:127--173, 2006.


\bibitem{PRS} A. Pizzo, D. Renfrew, A. Soshnikov
On Finite Rank Deformations of Wigner Matrices.
To appear in Annales de l'Institut Henri Poincar\'e (B) Probabilit\'es et Statistiques, 2011.

\bibitem{RaoSil09}
N.~R. {Rao} and J.~W. {Silverstein}.
\newblock {Fundamental limit of sample generalized eigenvalue based detection
  of signals in noise using relatively few signal-bearing and noise-only
  samples}.
\newblock {\em IEEE Journal of Selected Topics in Signal Processing},
  4(3): 468--480, 2010.


\bibitem{Remmert84}
Reinhold Remmert.
\newblock {\em Funktionentheorie. {I}}, volume~5 of {\em Grundwissen Mathematik
  [Basic Knowledge in Mathematics]}.
\newblock Springer-Verlag, Berlin, 1984.

\bibitem{Speicher93a}
R. Speicher.
\newblock Free convolution and the random sum of matrices.
\newblock {\em Publ. Res. Inst. Math. Sci.}, 29(5):731--744, 1993.

\bibitem{Voiculescu86}
D.~Voiculescu.
\newblock Addition of Certain Non commuting Random Variables
\newblock {\em J. Funct. Anal.}, 66:323--346, 1986.

\bibitem{Voiculescu91}
D.~V. Voiculescu.
\newblock Limit laws for random matrices and free products.
\newblock {\em Invent. Math.}, 104(1):201--220, 1991.

\bibitem{Voiculescu00}D.~V. Voiculescu.
\newblock The coalgebra of the free difference quotient and free probability.
\newblock {\em Internat. Math. Res. Notices}, (2):79--106, 2000.

\bibitem{VDN92}
D.~V. Voiculescu, K.~J. Dykema, and A.~Nica.
\newblock {\em Free random variables}, 
\newblock A noncommutative probability approach to free products with
  applications to random matrices, operator algebras and harmonic analysis on
  free groups.
volume~1 of {\em CRM Monograph Series}.
\newblock American Mathematical Society, Providence, RI, 1992.
\bibitem{Y}
Y. Q. Yin, Z. D. Bai, and P. R. Krishnaiah. 
On the limit of
the largest eigenvalue of the large-dimensional sample covariance
matrix. 
\newblock {\em Probab. Theory Related Fields}, 78(4): 509-521, 1988.
\end{thebibliography}
\end{document}